\documentclass[11pt]{article}
\usepackage{amsmath,amsthm, amssymb, amsfonts,accents,nccmath}
\usepackage{enumitem}
\usepackage{array}
\usepackage[toc,page]{appendix}
\usepackage[english]{babel}
\usepackage{dsfont}
\usepackage{booktabs}
\usepackage[linesnumbered,commentsnumbered,ruled,vlined]{algorithm2e}
\usepackage{tikz}
\usepackage{tikz-network}
\usepackage[utf8]{inputenc}
\usepackage[english]{babel}

\usepackage{caption}
\usepackage{subcaption}
 
\usepackage{hyperref}
\hypersetup{
    colorlinks=true,
    linkcolor=blue,
    filecolor=darkmagenta,      
    urlcolor=cyan,
    citecolor=red
} 
\makeatletter
\def\BState{\State\hskip-\ALG@thistlm}
\makeatother
\numberwithin{equation}{section}

\newtheorem{theorem}{Theorem}[section]
\newtheorem{cor}[theorem]{Corollary}
\newtheorem{lem}[theorem]{Lemma}
\newtheorem{prop}[theorem]{Proposition}

\usepackage{mathtools}

\newcommand{\bigzero}{\mbox{\normalfont\Large\bfseries 0}}

\newcommand{\ds}{\displaystyle}

\newcommand{\cof}{\textup{cof }}

\title{Inverse of the Squared Distance Matrix of a\\
 Complete Multipartite Graph}
\author{Joyentanuj Das\footnote{School of Informatics, The University of Edinburgh, Edinburgh EH8 9AB, Scotland, U.K.    Email: joyentanuj@gmail.com,  joyentanuj@ed.ac.uk}  \quad and \quad Sumit Mohanty\footnote{Indian Institute of Management Ranchi,  Prabandhan Nagar, Vill-Mudma, Nayasarai Road, Ranchi, Jharkhand-835303, India. \indent   Email:  sumitmath@gmail.com, sumit.mohanty@iimranchi.ac.in}}

\date{}

\begin{document}

\maketitle

\begin{abstract}
Let $G$ be a connected graph on  $n$ vertices and $d_{ij}$   be  the length of the shortest path between vertices $i$ and $j$ in $G$. We set $d_{ii}=0$ for every vertex $i$ in $G$.  The squared distance matrix $\Delta(G)$ of $G$ is the $n\times n$ matrix with $(i,j)^{th}$ entry equal to $0$ if $i = j$ and equal to $d_{ij}^2$ if $i \neq j$.  For a given complete $t$-partite graph $K_{n_1,n_2,\cdots,n_t}$ on $n=\sum_{i=1}^t n_i$ vertices, under some condition we find the inverse  $\Delta(K_{n_1,n_2,\cdots,n_t})^{-1}$ as a rank-one perturbation of a symmetric Laplacian-like matrix $\mathcal{L}$ with $\textup{rank} (\mathcal{L})=n-1$. We also investigate the inertia of $\mathcal{L}$.
\end{abstract}

\noindent {\sc\textsl{Keywords}:} Squared  distance matrix, Complete $t$-partite graphs,  Laplacian-like matrix,  Inverse, Inertia. 

\noindent {\sc\textbf{MSC}:}  05C12, 05C50

\section{Introduction  and Motivation}
Let $G$ be a  connected graph on vertices $1,2,\ldots,n$, and $d(i,j)$ be the length of the shortest path between vertices $i$ and $j$. We set $d(i,i)=0$ for $i=1,2,\dots,n.$ The distance matrix of a graph $G$ on $n$ vertices is an $n \times n$ matrix $D(G)=[d_{ij}]$, where $d_{ij}=d(i,j)$. In the literature, the  Hadamard product  $D(G)\circ D(G)$ is called the squared distance matrix of $G$ and is defined as $\Delta (G)=[d_{ij}^2]$. Thus, $\Delta (G)$ is a real symmetric matrix, and the eigenvalues of $\Delta (G)$ are real.

Before proceeding further, we now introduce a few notations. Let $I_n$ and $ \mathds{1}_n$ denote the identity matrix and the column vector of all ones,  respectively.   We use $\mathbf{0}_{m \times n}$ to represent zero matrix of order $m \times n$. We simply write $ \mathds{1}$,  $I$  and $\mathbf{0}$  if there is no scope of confusion with respect to the order. Further, $J_{m \times n}$  denotes the $m\times n$ matrix of all ones, and if $m=n$, we use the notation $J_m$. Given a matrix $A$, we write $A^t$  to denote the transpose of the matrix $A$. For a  symmetric  matrix $A$, the inertia of  $A$, denoted by $\textup{In}(A)$ is  the triplet  $(\mathbf{n}_{+}(A),\mathbf{n}_{0}(A), \mathbf{n}_{-}(A) )$, where $\mathbf{n}_{+}(A),\mathbf{n}_{0}(A)  \mbox{ and } \mathbf{n}_{-}(A)$ denote the number of positive eigenvalues of $A$, the multiplicity of  $0$  is an eigenvalue of $A$ and the number of negative eigenvalues of $A$, respectively. 

A well-known results  due to Graham and Pollak~\cite{Gr1}, if $T$  is a tree with $n$ vertices, then determinant of the distance matrix $D(T)$ is given by $\det D(T)=(-1)^{n-1}(n-1)2^{n-2}.$ Thus, the determinant does not depend on the structure of the tree but only on the number of vertices. Later in \cite{Gr2}, the inverse of $D(T)$ is obtained by Graham and Lov$\acute{\textup{a}}$sz as a rank one perturbation of the Laplacian matrix of $T$. These two results invoke significant interest, and several extensions and generalizations have been obtained (for example see~\cite{Bp3,JD1, Hou2,Zhou1}). Given a graph $G$, the primary objective of these results is to define a matrix  $\mathcal{L}$ called Laplacian-like matrix satisfying  $\mathcal{L}\mathds{1} = \mathbf{0}$ and $\mathds{1}^t \mathcal{L} = \mathbf{0}$ and find the inverse of the $D(G)$ as a rank one perturbation of $\mathcal{L}$. 

In~\cite{Bp4, Bp5}, Bapat and  Sivasubramanian first studied the squared distance matrix of a tree. To be specific, Bapat and  Sivasubramanian proved a formula for the determinant of the squared distance matrix $\Delta(T)$ of a tree $T$ in~\cite{Bp4}. Later, in~\cite{Bp5}, they obtained the inverse $\Delta(T)^{-1}$ (whenever it exists) as a rank one perturbation of a matrix and also computed the inertia of $\Delta(T)$. Bapat in~\cite{Bp5}, studied the determinant and the inverse of the squared distance matrix of a weighted tree. In~\cite{JD3}, Das and Mohanty studied the squared distance matrix of a complete multipartite graph  $K_{n_1,n_2,\cdots,n_t}$. To be precise, in~\cite{JD3}, authors first computed the inertia, energy of  $\Delta(K_{n_1,n_2,\cdots,n_t})$. Next, for fixed values of $n$ and $t$, they discussed the existence and the uniqueness of graphs for which the spectral radius and the energy of $\Delta(K_{n_1,n_2,\cdots,n_t})$ attained its maximum and minimum value. In~\cite{IM}, Mahato and Kannan obtained the determinant and the inverse of the squared distance matrix of a tree with matrix weights. In~\cite{How}, Howell, Kempton, Sandall and Sinkovic gave an alternative proof to obtain the inertia of the squared distance matrix of a tree due to Bapat and  Sivasubramanian in~\cite{Bp5} and also studied the inertia for a unicyclic graph. In the literature,  the spectral properties and the inverse of complete multipartite graphs has been studied with respect to the distance matrix and the adjacency matrix  (for example see~\cite{JD1,Delorme,Obo,Stev}).

In this article, we obtain a Laplacian-like matrix $\mathcal{L}$ with $\textup{rank} (\mathcal{L})= n-1$ for a given complete multipartite graph  $K_{n_1,n_2,\cdots,n_t}$, and prove that the inverse of  $\Delta(K_{n_1,n_2,\cdots,n_t})$ (whenever it exists) is a rank one perturbation of $\mathcal{L}$. In this process, we find a few interesting recurrence type relations involving $n_1,n_2,\cdots,n_t$. We also observe a few properties of $\mathcal{L}$,  compute the  $ \textup{In}(\mathcal{L})$ if $\det \Delta(K_{n_1,n_2,\cdots,n_t})\neq 0$ and give a  conjecture about the $\textup{In}(\mathcal{L}$) if $\det \Delta(K_{n_1,n_2,\cdots,n_t})=0$. 

This article is organized as follows. In Section~\ref{sec:prelim_results}, we prove a few preliminary results that are necessary for the subsequent sections. In Section~\ref{sec:inverse}, we compute the inverse of the squared distance matrix $\Delta(K_{n_1,n_2,\cdots,n_t})$ (whenever it exists).  Finally, in Section~\ref{sec:properties_of_L}, we obtain a few properties of $\mathcal{L}$ and investigate the inertia   $ \textup{In}(\mathcal{L})$.

\section{Some Preliminary Results}\label{sec:prelim_results}

Let $A$ be an $n \times n$ matrix. Let $A(i \mid j)$ be  the submatrix obtained by deleting the $i^{th}$ row and the $j^{th}$ column and for $1 \leq i,j \leq n$, the cofactor $c(i,j)$ is defined as $(-1)^{i+j} \det A(i \mid j)$. We use the notation  $\cof A$ to denote the sum of all cofactors of $A$ and state the following result.
\begin{lem}\label{Lem:cof}\cite{Bapat}
Let $A$ be an $n \times n$ matrix. Let $M$ be the matrix obtained from $A$ by subtracting the first row from all other rows and then subtracting the first column from all other columns. Then $$\cof A = \det M(1 | 1).$$
\end{lem}

The following is a standard result on computing the determinant of block matrices.

\begin{prop}\label{prop:blockdet}~\cite{Zhang}
 Let $A_{11}
   \mbox{ and } A_{22}$ be square matrices.  Then 
   \begin{small}
   $$\det \left[
\begin{array}{c|c}
A_{11} & \bigzero  \\
\midrule
A_{21} & A_{22}
\end{array}
\right] = \det A_{11} \det A_{22}.$$
   \end{small}   
\end{prop}

We will prove a lemma, which will help us to compute the  $\cof \Delta(K_{n_1,n_2,\cdots,n_t})$.

\begin{lem}\label{Lem:Det_3}
Let $C_m$ be an $m \times m$ matrix  of the following form
\begin{small}
\begin{equation*}\label{eqn:M_3}
C_m = \left[\begin{array}{ccccc}
n_1  & -4(n_1 - 1)& -4(n_1 - 1) & \cdots & -4(n_1 - 1)\\
n_2 & 2(n_2-2) & -n_2 &\cdots & -n_2\\
n_3 & -n_3 & 2(n_3-2) &\cdots & -n_3\\
\vdots & \vdots  &\vdots  & \ddots &\vdots\\
n_m  & -n_m & -n_m &  \cdots & 2(n_m-2)\\
\end{array} \right].
\end{equation*} 
\end{small}
The determinant of the above matrix is given by 
$$\det C_m = \sum_{i=1}^m \left( n_i \prod_{j=1\atop j \neq i}^m(3n_j - 4) \right).$$
\end{lem}

\begin{proof}
Adding the first column to all the remaining columns of $C_m$ yields the following matrix
\begin{small}
$$
\left[\begin{array}{ccccc}
n_1  & -3n_1+4 & -3n_1+4 & \cdots & -3n_1+4\\
n_2 & 3n_2-4 & 0 &\cdots & 0\\
n_3 & 0 & 3n_3-4  &\cdots & 0\\
\vdots & \vdots  &\vdots  & \ddots &\vdots\\
n_m  & 0 & 0 &  \cdots & 3n_m-4 \\
\end{array} \right].
$$
\end{small}
 Now, expanding along the first row, we get 
$$ \det C_m = n_1 \prod_{j=2}^m (3n_j-4)+(3n_1-4)\sum_{i=2}^m \left(n_i \prod_{j=2 \atop j\neq i}^m (3n_j-4)\right),$$
 and the desired result follows.
\end{proof}

Let $\Delta(K_{n_1,n_2,\cdots,n_t})$ be the squared distance matrix of the complete $t$-partite graph $K_{n_1,n_2,\cdots,n_t}$. Then $\Delta (K_{n_1,n_2,\cdots,n_t})$ can be expressed in the following block form 
\begin{small}
\begin{equation}\label{eqn:D(K)}
\Delta(K_{n_1,n_2,\cdots,n_t}) = \left[\begin{array}{c|c|c|c}
4(J_{n_1} - I_{n_1}) & J_{n_1 \times n_2} & \cdots & J_{n_1 \times n_t}\\
\hline
J_{n_2 \times n_1} &4(J_{n_2} - I_{n_2}) & \cdots & J_{n_2 \times n_t}\\
\hline
\vdots & \cdots  & \ddots &\vdots\\
\hline
J_{n_t \times n_1} & J_{n_t \times n_2} &  \cdots & 4(J_{n_t} - I_{n_t})\\
\end{array} \right].
\end{equation}
\end{small}
Throughout this article we assume that the vertices of $\Delta(K_{n_1,n_2,\cdots,n_t})$ are indexed as in  Eqn.~\eqref{eqn:D(K)}.

\begin{theorem}\label{thm:cof}
Let $\Delta(K_{n_1,n_2,\cdots,n_t})$ be the squared distance matrix of the complete $t$-partite graph $K_{n_1,n_2,\cdots,n_t}$ on $n=\sum_{i=1}^t n_i$ vertices. Then, the sum of the cofactors  of the squared distance matrix is given by
$$\cof \Delta(K_{n_1,n_2,\cdots,n_t}) = (-4)^{n-t}\left[ \sum_{i=1}^t \left( n_i \prod_{j=1\atop j \neq i}^t (3n_j - 4) \right)\right].$$
\end{theorem}

\begin{proof}
For complete $t$-partite graph $K_{n_1,n_2,\cdots,n_t}$ with $n_i=1$ for all $1\leq i\leq t$, we have $K_{n_1,n_2,\cdots,n_t}=K_t$. Then, $\Delta(K_t)=D(K_t)=J_n-I_n$ and the result is true as $\cof \Delta(K_t)=\cof D(K_t)= (-1)^{t-1} t$. For other cases, without loss of generality, let $n_1>1$ and  $M$ be the matrix obtained from $\Delta(K_{n_1,n_2,\cdots,n_t})$ by subtracting the first row from all other rows and then subtracting the first column from all other columns. Then the block form of the matrix $M(1|1)$ is given by
$$
\begin{small}
 \left[\begin{array}{c|c|c|c}
-4(J_{n_1-1} + I_{n_1-1}) & -4J_{(n_1-1) \times n_2} & \cdots & -4J_{(n_1-1) \times n_t}\\
\hline
-4J_{n_2 \times (n_1-1)} &  2J_{n_2}-4I_{n_2} & \cdots & -J_{n_2 \times n_t}\\
\hline
\vdots & \cdots  & \ddots &\vdots\\
\hline
-4J_{n_t \times (n_1-1)} & -J_{n_t \times n_2} &  \cdots &  2J_{n_t}-4I_{n_t}\\
\end{array} \right].
\end{small}
$$
First, for each partition of $M(1|1)$, we subtract the first column from all other columns, then add all the rows to the first row. Further, we shift the first column of each of the $t$-partitions to the first  $t$ columns and repeat the same operation for the rows. Then, the resulting matrix is of the following block form
\begin{small}
$$
\left[
\begin{array}{c|c}
\widetilde{C}_t& \mathbf{0} \\
\hline
* & -4I_{n-(t+1)}
\end{array}
\right], 
$$ where
$$
\widetilde{C}_t = \left[\begin{array}{ccccc}
-4n_1  & -4(n_1 - 1)& -4(n_1 - 1) & \cdots & -4(n_1 - 1)\\
-4n_2 & 2(n_2-2) & -n_2 &\cdots & -n_2\\
-4n_3 & -n_3 & 2(n_3-2) &\cdots & -n_3\\
\vdots & \vdots  &\vdots  & \ddots &\vdots\\
-4n_t & -n_t & -n_t &  \cdots & 2(n_t-2)\\
\end{array} \right].$$
\end{small}
Using the Proposition~\ref{prop:blockdet} and  Lemma~\ref{Lem:Det_3}, the result follows. 
\end{proof}
We now recall a result that gives a formula to obtain the determinant of $\Delta(K_{n_1,n_2,\cdots,n_t})$.
\begin{theorem}~\cite[Corollary~$3.5$]{JD3}\label{thm:detD_single}
Let $\Delta(K_{n_1,n_2,\cdots,n_t})$ be the squared distance matrix of complete $t$-partite graph $K_{n_1,n_2,\cdots,n_t}$ on $n=\sum_{i=1}^t n_i$ vertices. Then, the determinant of the squared distance matrix is given by
$$\det \Delta(K_{n_1,n_2,\cdots,n_t}) = (-4)^{n-t}\left[ \sum_{i=1}^t \left( n_i \prod_{j=1\atop j \neq i}^t (3n_j - 4) \right) +  \prod_{i=1}^t (3n_i - 4)\right].$$ 
\end{theorem}

The following result discusses the cases in which the determinant and the sum of the cofactors of the squared distance matrix of $K_{n_1,n_2, \cdots, n_t}$ are zero. 

\begin{theorem}\label{thm:det-cof-zero}
Let $G$ be a complete $t$-partite graph $K_{n_1,n_2,\cdots,n_t}$ on $n=\sum_{i=1}^t n_i$ vertices and $\Delta(G)$ be the squared distance matrix of $G$. If $h=|\{i: n_i=1\}|$,  then the following results hold.
\begin{enumerate}
\item[$(i)$]  $\det \Delta(G)=0$ only if $\ds  \dfrac{t}{4}+ \dfrac{3}{4} < h \leq \dfrac{t}{2}+ \dfrac{1}{2}$. Furthermore, let $t=s+h$ such that  $n_i\geq 2$ for $1\leq i\leq s$ and $n_i=1$ for $s+1 \leq i\leq t=s+h$.  Then, $\det \Delta(G)=0$ if and only if  $h-1 = \ds\sum_{i=1}^s \dfrac{n_i}{3n_i-4}.$ 

\item[$(ii)$]$\cof \Delta(G)=0$ only if $\ds  \dfrac{t}{4} < h \leq \dfrac{t}{2}$.  Furthermore, let $t=s+h$ such that  $n_i\geq 2$ for $1\leq i\leq s$ and $n_i=1$ for $s+1 \leq i\leq t=s+h$. Then, $\cof \Delta(G)=0$ if and only if  $h= \ds\sum_{i=1}^s \dfrac{n_i}{3n_i-4}$.
\end{enumerate}
\end{theorem}
\begin{proof}
By Theorems~\ref{thm:cof} and ~\ref{thm:detD_single}, it is easy to see that $\det \Delta (G)$ and $\cof \Delta (G)$ are non-zero if $n_i\geq 2$ for $1\leq i \leq t$. Thus, $\det \Delta (G)$ and $\cof \Delta (G)$ are zero only if some of the $n_i$'s are one. Let $t=s+h$ and   $n_1,n_2,\cdots,n_t$ be positive integers  such that  $n_i\geq 2$ for $1\leq i\leq s$ and $n_i=1$ for $s+1 \leq i\leq t=s+h$.

Now substituting $n_i=1$ for $s+1 \leq i\leq t=s+h$ in  Theorem~\ref{thm:detD_single}, the determinant of the squared distance matrix $\Delta(G)$ is given by
\begin{align*}
\det \Delta(G) &= (-4)^{n-t}\left[ \sum_{i=1}^t \Bigg( n_i \prod_{j=1\atop j \neq i}^t (3n_j - 4) \Bigg) +  \prod_{i=1}^t (3n_i - 4)\right]\nonumber\\
               &= (-4)^{n-t}\left[ (-1)^h\sum_{i=1}^s \Bigg( n_i \prod_{j=1 \atop j \neq i}^s (3n_j - 4) \Bigg) + (-1)^{h-1} h \prod_{i=1}^s (3n_i - 4)+ (-1)^h \prod_{i=1}^s (3n_i - 4)\right]\nonumber\\
               &=(-4)^{n-t} (-1)^{h-1} \left[ (h-1)\prod_{i=1}^s (3n_i - 4) - \sum_{i=1}^s \Bigg( n_i  \prod_{j=1\atop j \neq i}^s (3n_j - 4) \Bigg)\right]\nonumber\\
               &=(-4)^{n-t} (-1)^{h-1} \left[ (h-1)- \sum_{i=1}^s \frac{n_i}{3n_i - 4}\right] \prod_{i=1}^s (3n_i - 4).
\end{align*}
 Since $n_i\geq 2$ for $1\leq i\leq s$, so  $\prod_{i=1}^s (3n_i - 4)>0.$  Thus $\det \Delta(G)=0$ if and only if  $h-1 = \ds\sum_{i=1}^s \dfrac{n_i}{3n_i-4}$.   Further,  $h-1 = \ds\sum_{i=1}^s \dfrac{n_i}{3n_i-4}  = \dfrac{s}{3} + \dfrac{4}{3} \ds\sum_{i=1}^s \dfrac{1}{3n_i-4}= \dfrac{t-h}{3} + \dfrac{4}{3} \ds\sum_{i=1}^s \dfrac{1}{3n_i-4}$, which implies that  $ \ds\sum_{i=1}^s \dfrac{1}{3n_i-4}=h-\dfrac{t+3}{4}$.  Using $n_i\geq 2$ for $1\leq i\leq s$, we have $ 0<\ds\sum_{i=1}^s \dfrac{1}{3n_i-4} \leq \dfrac{s}{2}=\dfrac{t-h}{2}$. Therefore, $0< h-\dfrac{t+3}{4} \leq \dfrac{t-h}{2}$ and hence $\dfrac{t}{4}+ \dfrac{3}{4} < h \leq \dfrac{t}{2}+ \dfrac{1}{2}.$ This proves part $(i)$.
 
Next, substituting $n_i=1$ for $s+1 \leq i\leq t=s+h$ in  Theorem~\ref{thm:cof}, the sum of the cofactors  of the squared distance matrix $\Delta(G)$ is given by 
\begin{align*}\label{eqn:cof-h-1}
\cof \Delta(G) &= (-4)^{n-t}\left[ (-1)^h\sum_{i=1}^s \Bigg( n_i  \prod_{j=1 \atop j \neq i}^s (3n_j - 4) \Bigg) + (-1)^{h-1} h \prod_{i=1}^s (3n_i - 4)\right]\nonumber\\
 &=(-4)^{n-t} (-1)^{h-1} \left[ h- \sum_{i=1}^s \frac{n_i}{3n_i - 4}\right] \prod_{i=1}^s (3n_i - 4).
 \end{align*}
Since $n_i\geq 2$ for $1\leq i\leq s$, $\cof \Delta(G)=0$ if and only if  $h= \ds\sum_{i=1}^s \dfrac{n_i}{3n_i-4}$. Moreover, arguing similar to part $(i)$,   $h= \ds\sum_{i=1}^s \dfrac{n_i}{3n_i-4}$ yields that  $ \ds\sum_{i=1}^s \dfrac{1}{3n_i-4}=h-\dfrac{t}{4}$. Using $n_i\geq 2$ for $1\leq i\leq s$, we have $ 0<\ds\sum_{i=1}^s \dfrac{1}{3n_i-4} \leq \dfrac{s}{2}=\dfrac{t-h}{2}$ implies that  $0< h-\dfrac{t}{4} \leq \dfrac{t-h}{2}$ and hence $\dfrac{t}{4} < h \leq \dfrac{t}{2}.$ This completes the proof.
\end{proof}

\section{Inverse of $\Delta(K_{n_1,n_2,\cdots,n_t})$}\label{sec:inverse}
In this section, we first find the inverse of the squared distance matrix $\Delta(K_{n_1,n_2,\cdots,n_t})$ as a rank one perturbation of a Laplacian-like matrix subject to the condition that  $\cof \Delta(K_{n_1,n_2,\cdots,n_t})\neq 0.$ Notice that by~\cite[Lemma $4.13$]{JD1}, it is known that  $\cof \Delta(K_{n_1,n_2,\cdots,n_t})\neq 0$    is a necessary condition to present the inverse of $\Delta(K_{n_1,n_2,\cdots,n_t})$ (whenever it exists) as a rank one perturbation of a matrix.   Next, whenever it exists, we provide a formula for the inverse of  $\Delta(K_{n_1,n_2,\cdots,n_t})$. Before proceeding further, we introduce a few notations useful for the subsequent results. Let $n_i\in \mathbb{N}$ for $1\leq i\leq t$ and  $t\geq 2$, let us denote 
\begin{equation}\label{eqn:phi}
\hspace*{-2 cm}\begin{cases}
\Phi_{n_1,n_2,\cdots,n_t}= \ds \prod_{i=1}^t (3n_i-4),\\
\Phi_{\widehat{n_i}}= \Phi_{n_1, n_2,\cdots n_{i-1},n_{i+1},\cdots, n_t}  = \ds \prod_{k=1\atop k\neq i}^t (3n_k-4),\\
\Phi_{\widehat{n_i,n_j}}= \Phi_{n_1, n_2,\cdots n_{i-1},n_{i+1},\cdots, n_{j-1},n_{j+1},\cdots,  n_t}  = \ds \prod_{k=1\atop k\neq i,j}^t (3n_k-4),\\
\end{cases}
\end{equation}

\begin{equation}\label{eqn:psi}
\begin{cases}
\Psi_{n_1,n_2,\cdots,n_t}= \ds  \sum_{i=1}^t \left(n_i \prod_{ j=1\atop j\neq i}^t (3n_j-4)\right),\\
\Psi_{\widehat{n_i}}= \Psi_{n_1, n_2,\cdots n_{i-1},n_{i+1},\cdots, n_t}  = \ds  \sum_{k=1\atop k\neq i}^t \left(n_k \prod_{ l=1 \atop l\neq i}^t (3n_l-4)\right),\\
\Psi_{\widehat{n_i,n_j}}= \Psi_{n_1, n_2,\cdots n_{i-1},n_{i+1},\cdots, n_{j-1},n_{j+1},\cdots,  n_t}  = \ds  \sum_{k=1\atop k\neq i,j}^t \left(n_k \prod_{ l=1 \atop l\neq i,j}^t (3n_l-4)\right),\\
\end{cases}
\end{equation}
and
\begin{equation}\label{eqn:theta}
\hspace*{-5 cm} \begin{cases}
\Theta_{n_1,n_2,\cdots,n_t}= \Phi_{n_1,n_2,\cdots,n_t}+ \Psi_{n_1,n_2,\cdots,n_t},\\
\Theta_{\widehat{n_i}}= \Phi_{\widehat{n_i}}+ \Psi_{\widehat{n_i}}.
\end{cases}
\end{equation}

Next, we state a few identities based on the notations from  Eqns.~\eqref{eqn:phi} -~\eqref{eqn:theta} and  are useful for our subsequent calculations. The proofs of these identities are mostly computational and hence omitted (for details of the proofs see Appendix~\ref{apn_1}).
\begin{lem}\label{lem:identities}
Let $n_i\in \mathbb{N}$ for $1\leq i\leq t$ and  $t\geq 2$, we have the following identities.
\begin{enumerate}
\item[(a)] $\ds \Theta_{n_1,n_2,\cdots,n_t}=  (3n_i-4)\Phi_{\widehat{n_i}} + \sum_{k=1}^t n_k \Phi_{\widehat{n_k}}.$

\item [(b)] $  \ds \Theta_{n_1,n_2,\cdots,n_t}=  (3n_i-4)\Theta_{\widehat{n_i}}+ n_i \Phi_{\widehat{n_i}}.$

\item [(c)] $\ds \Psi_{n_1,n_2,\cdots,n_t}=  (3n_i-4)\Psi_{\widehat{n_i}} +  n_i \Phi_{\widehat{n_i}}. $

\item [(d)] $\ds \Psi_{\widehat{n_i}}
=\sum_{k=1\atop k\neq i}^t n_k  \Phi_{\widehat{n_i,n_k}} =(3n_j-4)\Psi_{\widehat{n_i,n_j}}+n_j \Phi_{\widehat{n_i,n_j}}.$
\end{enumerate}
\end{lem}

Let $G = (V,E)$ be a complete $t$-partite graph $K_{n_1,n_2,\cdots,n_t}$  on $n=\sum_{i=1}^t n_i$ vertices such that the vertex $V$ is partitioned into $t$ subsets  $V_i$ for $1 \leq i \leq t$ and $|V_i| = n_i$. Let $\Delta(G)$ be the squared distance matrix of $G$.  Using notations in  Eqns.~\eqref{eqn:phi} -~\eqref{eqn:theta}, by  Theorem~\ref{thm:detD_single}, we have $\det \Delta(G)=0$ if and only if $\Theta_{n_1,n_2,\cdots,n_t} =0$ and by Theorem~\ref{thm:cof}, we have  $\cof \Delta(G)=0$ if and only if $\Psi_{n_1,n_2,\cdots,n_t} =0$. With the above notations and observations,  we define a few parameters useful to find the inverse of $\Delta(G)$. If $\cof \Delta(G) \neq 0$, then we define the constant and an $n$-dimensional column vector $\nu$ as follows: 
\begin{equation}\label{eqn:lambda_sq_multi1}
\ds \lambda = \frac{\det \Delta(G)}{\cof \Delta(G)}=\frac{\Theta_{n_1,n_2,\cdots,n_t}} {\Psi_{n_1,n_2,\cdots,n_t}},
\end{equation} 
and   
\begin{equation}\label{eqn:nu_sq_multi}
\nu(v) = \frac{1}{\Psi_{n_1,n_2,\cdots,n_t}} \left( \sum_{i=1}^t \sum_{v \in V_i} \Phi_{\widehat{n_i}} \right), \text{ where } v \in V.
\end{equation}
Next,  we define  the Laplacian-like matrix  $\mathcal{L}$ satisfying  $\mathcal{L}\mathds{1} = \mathbf{0}$ and $\mathds{1}^t \mathcal{L} = \mathbf{0}$ with respect to the squared distance matrix $\Delta(G)$ with $\cof \Delta(G) \neq 0$ as follows: Let $\mathcal{L} = [\mathcal{L}_{uv}]_{u,v \in V}$, where 
\begin{equation}\label{eqn:lap_like_sq_multi}
\mathcal{L}_{uv}=
\begin{cases}\vspace*{.28cm}
\dfrac{1}{2\Psi_{n_1,n_2,\cdots,n_t}}\left[ \dfrac{n_i-1}{2} \Theta_{\widehat{n_i}} +(n_i-3) \Psi_{\widehat{n_i}}\right]=a_i & \text{ if } u=v \text{ and } u,v \in V_i,\\\vspace*{.28cm}
-\dfrac{1}{2\Psi_{n_1,n_2,\cdots,n_t}}\left[\dfrac{1}{2} \Theta_{\widehat{n_i}} +\Psi_{\widehat{n_i}}\right] = b_i & \text{ if } u \neq v \text{ and } u,v \in V_i,\\
\dfrac{1}{\Psi_{n_1,n_2,\cdots,n_t}}\Phi_{\widehat{n_i,n_j}} = c_{ij} & \text{ if } u \neq v, u \in V_i \text{ and } v \in V_j.
\end{cases}
\end{equation}
 In view of Eqns.~\eqref{eqn:D(K)} and \eqref{eqn:lap_like_sq_multi},  the product of matrices $\mathcal{L}$ and $\Delta(G)$ is given by  $\mathcal{L}\Delta(G)=[(\mathcal{L}\Delta(G))_{uv}]_{u,v \in V}$, where
\begin{equation}\label{eqn:L_Delta}
(\mathcal{L}\Delta(G))_{uv}=
\begin{cases}\vspace*{.28cm}
4(n_i-1)b_i + \ds \sum_{k=1\atop k\neq i}^t n_k c_{ik} & \text{ if } u=v \text{ and } u,v \in V_i,\\\vspace*{.28cm}
4a_i + 4(n_i-2)b_i+ \ds \sum_{k=1\atop k\neq i}^t n_k c_{ik} & \text{ if } u \neq v \text{ and } u,v \in V_i,\\
a_i + (n_i-1)b_i + 4(n_j-1)c_{ij} +\ds  \sum_{k=1\atop k\neq i,j}^tn_k c_{ik} & \text{ if } u \neq v, u \in V_i \text{ and } v \in V_j.
\end{cases}
\end{equation}

\begin{lem}\label{lem:sq-nu-1}
Let $G$ be a complete $t$-partite graph $K_{n_1,n_2,\cdots,n_t}$.   Let $\Delta(G)$ be the squared distance matrix of $G$ and $\cof \Delta(G) \neq 0$. Then $\Delta(G)\nu = \lambda \mathds{1}$, where $\lambda$ and $\nu$ is as defined in Eqns.~\eqref{eqn:lambda_sq_multi1} and~\eqref{eqn:nu_sq_multi}, respectively.
\end{lem}
\begin{proof}
Let $G = (V,E)$ be the complete $t$-partite graph $K_{n_1,n_2,\cdots,n_t}$  on $n=\sum_{i=1}^t n_i$ vertices such that the vertex $V$ is partitioned into $t$ subsets  $V_i$ for $1 \leq i \leq t$ and $|V_i| = n_i$. Let $\eta=\Delta(G)\nu$. We will show  $\eta(v)=\lambda$ for all $v\in V.$ For $v\in V_i$, we have
\begin{align*} 
\eta(v)&= 4(n_i-1)\dfrac{\Phi_{\widehat{n_i}}}{\Psi_{n_1,n_2,\cdots,n_t}}+\sum_{k=1\atop k \neq i}^t n_k \dfrac{\Phi_{\widehat{n_k}}}{\Psi_{n_1,n_2,\cdots,n_t}}\\
&= \dfrac{1}{\Psi_{n_1,n_2,\cdots,n_t}} \left[ (3n_i-4)\Phi_{\widehat{n_i}} + \sum_{k=1}^t n_k \Phi_{\widehat{n_k}}  \right]\\
&= \frac{\Theta_{n_1,n_2,\cdots,n_t}} {\Psi_{n_1,n_2,\cdots,n_t}}=\lambda.
\end{align*} 
This completes the proof.
\end{proof}

\begin{lem}\label{lem:Lsq+I=nu}
Let $G$ be a complete $t$-partite graph $K_{n_1,n_2,\cdots,n_t}$. Let $\Delta(G)$ be the squared distance matrix of $G$ and $\cof \Delta(G) \neq 0$. If  $\mathcal{L}$ is the Laplacian-like matrix  defined in Eqn.~\eqref{eqn:lap_like_sq_multi},  then $\mathcal{L}\Delta(G) + I= \nu \mathds{1}^t$, where $\nu$ is as defined in Eqn.~\eqref{eqn:nu_sq_multi}.
\end{lem}
\begin{proof}
Let $G = (V,E)$ be the complete $t$-partite graph $K_{n_1,n_2,\cdots,n_t}$  on $n=\sum_{i=1}^t n_i$ vertices such that the vertex set $V$ is partitioned into $t$ subsets  $V_i$ for $1 \leq i \leq t$ and $|V_i| = n_i$. We will use Eqns.~\eqref{eqn:phi} -~\eqref{eqn:theta}, identities of Lemma~\ref{lem:identities}, Eqn.~\eqref{eqn:L_Delta}, and  consider the following cases to complete the proof.\\

\noindent \underline{\textbf{Case 1:}} Let $u=v$, where $u,v \in V_i$ for $1\leq i\leq t$. 
\begin{flalign*}
\hspace*{1cm}(\mathcal{L}\Delta(G) + I)_{uv} &= 1 + 4(n_i-1)b_i + \sum_{k=1\atop k\neq i}^t n_k c_{ik}&&\\
&=1-\dfrac{n_i-1}{\Psi_{n_1,n_2,\cdots,n_t}}\left[ \Theta_{\widehat{n_i}} +2 \Phi_{\widehat{n_i}}\right]+ \dfrac{\Psi_{\widehat{n_i}}}{\Psi_{n_1,n_2,\cdots,n_t}}&&\\
&=1-\dfrac{(n_i-1)\left[ 3 \Psi_{\widehat{n_i}}+\Phi_{\widehat{n_i}} \right] + \Psi_{\widehat{n_i}}}{\Psi_{n_1,n_2,\cdots,n_t}}&&\\
&=1-\dfrac{\left[ (3n_i-4) \Psi_{\widehat{n_i}}+n_i\Phi_{\widehat{n_i}} \right] - \Phi_{\widehat{n_i}}}{\Psi_{n_1,n_2,\cdots,n_t}}&&\\
&=1-\dfrac{\Psi_{n_1,n_2,\cdots,n_t} - \Phi_{\widehat{n_i}}}{\Psi_{n_1,n_2,\cdots,n_t}}= \dfrac{\Phi_{\widehat{n_i}}}{\Psi_{n_1,n_2,\cdots,n_t}}=\nu(u).&&
\end{flalign*}

\noindent \underline{\textbf{Case 2:}} Let $u \neq v,$ where $u,v \in V_i$ for $1\leq i\leq t$. 
\begin{flalign*}
\hspace*{1cm}(\mathcal{L}\Delta(G) + I)_{uv} &= 4a_i + 4(n_i-2)b_i+ \sum_{k=1\atop k\neq i}^t n_k c_{ik}&&\\
&=\dfrac{(n_i-1)\Theta_{\widehat{n_i}}+2(n_i-3)\Psi_{\widehat{n_i}}}{\Psi_{n_1,n_2,\cdots,n_t}}- \dfrac{(n_i-2)\Theta_{\widehat{n_i}}+2(n_i-2)\Psi_{\widehat{n_i}}}{\Psi_{n_1,n_2,\cdots,n_t}}+ \dfrac{\Psi_{\widehat{n_i}}}{\Psi_{n_1,n_2,\cdots,n_t}}&&\\
&=\dfrac{\Theta_{\widehat{n_i}}+[2(n_i-3)-2(n_i-2)+1]\Psi_{\widehat{n_i}}}{\Psi_{n_1,n_2,\cdots,n_t}}&&\\
&=\dfrac{\Theta_{\widehat{n_i}}-\Psi_{\widehat{n_i}}}{\Psi_{n_1,n_2,\cdots,n_t}} = \dfrac{\Phi_{\widehat{n_i}}}{\Psi_{n_1,n_2,\cdots,n_t}}=\nu(u).&&
\end{flalign*}

\noindent \underline{\textbf{Case 3:}} Let $ u \in V_i \text{ and } v \in V_j$ for $1\leq i,j\leq t$ with $i\neq j$.
\begin{flalign*}
\hspace*{1cm}(\mathcal{L}\Delta(G) + I)_{uv} &= a_i + (n_i-1)b_i + 4(n_j-1)c_{ij} + \sum_{k=1\atop k\neq i,j}^tn_k c_{ik}&&\\
&=\dfrac{(n_i-1)\Theta_{\widehat{n_i}}+ 2(n_i-3)\Psi_{\widehat{n_i}}}{4\Psi_{n_1,n_2,\cdots,n_t}}-
\dfrac{(n_i-1)\Theta_{\widehat{n_i}}+ 2(n_i-1)\Psi_{\widehat{n_i}}}{4\Psi_{n_1,n_2,\cdots,n_t}}&&\\
&\hspace*{6cm}+\dfrac{4(n_j-1)\Phi_{\widehat{n_i,n_j}}}{\Psi_{n_1,n_2,\cdots,n_t}}+\dfrac{(3n_j-4)\Psi_{\widehat{n_i,n_j}}}{\Psi_{n_1,n_2,\cdots,n_t}}&&\\
&= -\dfrac{\Psi_{\widehat{n_i}}}{\Psi_{n_1,n_2,\cdots,n_t}}+\dfrac{4(n_j-1)\Phi_{\widehat{n_i,n_j}}}{\Psi_{n_1,n_2,\cdots,n_t}}+\dfrac{(3n_j-4)\Psi_{\widehat{n_i,n_j}}}{\Psi_{n_1,n_2,\cdots,n_t}}&&\\
&= -\dfrac{\Psi_{\widehat{n_i}}}{\Psi_{n_1,n_2,\cdots,n_t}}+\dfrac{\left[(3n_j-4)\Psi_{\widehat{n_i,n_j}}+n_j \Phi_{\widehat{n_i,n_j}}\right] +(3n_j-4)\Phi_{\widehat{n_i,n_j}}}{\Psi_{n_1,n_2,\cdots,n_t}}&&\\
&= -\dfrac{\Psi_{\widehat{n_i}}}{\Psi_{n_1,n_2,\cdots,n_t}}+\dfrac{\Psi_{\widehat{n_i}} +\Phi_{\widehat{n_i}}}{\Psi_{n_1,n_2,\cdots,n_t}} = \dfrac{\Phi_{\widehat{n_i}}}{\Psi_{n_1,n_2,\cdots,n_t}}=\nu(u).&&
\end{flalign*}
This completes the proof.
\end{proof}

We now show that if $\Delta(K_{n_1,n_2,\cdots,n_t})$ is invertible, then we find the inverse of the squared distance matrix $\Delta(K_{n_1,n_2,\cdots,n_t})$ as a rank one perturbation of a Laplacian-like matrix $\mathcal{L}$  subject to the condition that  $\cof \Delta(K_{n_1,n_2,\cdots,n_t})\neq 0$.

\begin{theorem}\label{thm:Delta_inv}
Let $G$ be the complete $t$-partite graph $K_{n_1,n_2,\cdots,n_t}$ and  $\Delta(G)$ be the squared distance matrix of $G$. If $\det \Delta(G) \neq 0$ and $\cof \Delta(G) \neq 0$, then $$\Delta(G)^{-1} = - \mathcal{L} + \frac{1}{\lambda} \nu \nu^t,$$ 
where $\mathcal{L}$ is the Laplacian-like matrix  defined in Eqn.~\eqref{eqn:lap_like_sq_multi}, $\lambda$ and $\nu$ is as defined in Eqns.~\eqref{eqn:lambda_sq_multi1} and~\eqref{eqn:nu_sq_multi}, respectively.
\end{theorem}

\begin{proof}
Using Lemma~\ref{lem:sq-nu-1}, we have $\nu^t \Delta(G)=\lambda \mathds{1}^t$, which implies that $\nu \nu^t \Delta(G) = \lambda \nu \mathds{1}^t$. By Eqn.~\eqref{eqn:lambda_sq_multi1}, $\det \Delta(G)\neq 0$ if and only if $\lambda\neq 0$, and hence using Lemma~\ref{lem:Lsq+I=nu}, we get $ \mathcal{L}\Delta(G)+I=\nu \mathds{1}^t =\dfrac{1}{\lambda} \nu \nu^t \Delta(G)$.  Therefore, $\Delta(G)^{-1} = - \mathcal{L} + \dfrac{1}{\lambda} \nu \nu^t$.
\end{proof}

We conclude this section with the result that gives a block matrix form for the inverse of $\Delta(K_{n_1,n_2,\cdots,n_t})$, whenever it exists.

\begin{theorem}
Let $G$ be a complete $t$-partite graph $K_{n_1,n_2,\cdots,n_t}$ and  $\Delta(G)$ be the squared distance matrix of $G$. If $\det  \Delta(G)\neq 0$, then the inverse in $t \times t$ block form is given by $\Delta(G)^{-1}=[X_{ij}]$, where
\begin{equation}\label{eqn:Delta_inv}
X_{ij}=
\begin{cases}
\left(\dfrac{3 \Theta_{\widehat{n_i}} + \Phi_{\widehat{n_i}}}{4 \Theta_{n_1, n_2,\cdots,n_t}} \right)J_{n_i} - \dfrac{1}{4} I_{n_i} & \textup{if} \ i = j,\\
\\
-\dfrac{\Phi_{\widehat{n_i,n_j}}}{\Theta_{n_1, n_2,\cdots,n_t}} J_{n_i \times n_j} & \textup{if} \ i \neq j.\\
\end{cases}
\end{equation}
\end{theorem}

\begin{proof}
Let $ \Delta(G)=[\Delta_{ij}]$ be the $t \times t$ block form of the squared distance matrix of $G$ in Eqn.~\eqref{eqn:D(K)}. and  $X=[X_{ij}]$ be a $t \times t$ block matrix, where $X_{ij}$ defined in Eqn.~\eqref{eqn:Delta_inv}. Consider the block matrix $Y=  \Delta(G)X= [Y_{ij}]$, where   $Y_{ij}= \sum_{k=1}^t \Delta_{ik} X_{kj}$ for $1\leq i,j \leq t$. We will use Eqns.~\eqref{eqn:phi} -~\eqref{eqn:theta} and identities of Lemma~\ref{lem:identities} to show that $Y=I$ to complete the proof.\\

\noindent \underline{\textbf{Case 1:}} For $i=j$; $1\leq i \leq t$.
\begin{align*}
Y_{ii} &= (J_{n_i}-I_{n_i}) \left[ \left( \dfrac{3 \Theta_{\widehat{n_i}} + \Phi_{\widehat{n_i}}}{ \Theta_{n_1, n_2,\cdots,n_t}} \right)J_{n_i} -  I_{n_i} \right] - \sum_{k=1 \atop k \neq i}^t \dfrac{\Phi_{\widehat{n_kn_i}}}{\Theta_{n_1, n_2,\cdots,n_t}} J_{n_i \times n_k}  J_{n_k \times n_i}\\
&=  \dfrac{3 \Theta_{\widehat{n_i}} + \Phi_{\widehat{n_i}}}{\Theta_{n_1, n_2,\cdots,n_t}} \times n_i J_{n_i} - J_{n_i} - \dfrac{3 \Theta_{\widehat{n_i}} + \Phi_{\widehat{n_i}}}{\Theta_{n_1, n_2,\cdots,n_t}} \times J_{n_i} + I_{n_i} - \sum_{k=1 \atop k \neq i}^t \dfrac{\Phi_{\widehat{n_kn_i}}}{\Theta_{n_1, n_2,\cdots,n_t}} \times n_k J_{n_i}\\
&=\left[\dfrac{3n_i \Theta_{\widehat{n_i}} + n_i \Phi_{\widehat{n_i}}}{\Theta_{n_1, n_2,\cdots,n_t}} - \dfrac{3 \Theta_{\widehat{n_i}} + \Phi_{\widehat{n_i}}}{ \Theta_{n_1, n_2,\cdots,n_t}} - 1 \right]J_{n_i} + I_{n_i}- \sum_{k=1 \atop k \neq i}^t  \dfrac{n_k \Phi_{\widehat{n_kn_i}}}{\Theta_{n_1, n_2,\cdots,n_t}}   J_{n_i}\\
&=\left[\dfrac{3n_i \Theta_{\widehat{n_i}} + n_i \Phi_{\widehat{n_i}}}{\Theta_{n_1, n_2,\cdots,n_t}} - \dfrac{3 \Theta_{\widehat{n_i}} + \Phi_{\widehat{n_i}}}{ \Theta_{n_1, n_2,\cdots,n_t}} - 1 \right]J_{n_i} + I_{n_i}-   \dfrac{\Psi_{\widehat{n_i}}}{\Theta_{n_1, n_2,\cdots,n_t}} J_{n_i}\\
&=\left[\dfrac{3n_i \Theta_{\widehat{n_i}} + n_i \Phi_{\widehat{n_i}}}{\Theta_{n_1, n_2,\cdots,n_t}} - \dfrac{3 \Theta_{\widehat{n_i}} + \Phi_{\widehat{n_i}}}{ \Theta_{n_1, n_2,\cdots,n_t}} - 1 \right]J_{n_i} + I_{n_i}-   \dfrac{\Theta_{\widehat{n_i}}- \Phi_{\widehat{n_i}} }{\Theta_{n_1, n_2,\cdots,n_t}} J_{n_i}\\
&= \left[\dfrac{3n_i \Theta_{\widehat{n_i}} + n_i \Phi_{\widehat{n_i}}}{\Theta_{n_1, n_2,\cdots,n_t}} - \dfrac{3 \Theta_{\widehat{n_i}} + \Phi_{\widehat{n_i}}}{ \Theta_{n_1, n_2,\cdots,n_t}} - \dfrac{\Theta_{\widehat{n_i}} - \Phi_{\widehat{n_i}}}{\Theta_{n_1, n_2,\cdots,n_t}} - 1 \right]J_{n_i} + I_{n_i}\\
&= \left[\dfrac{(3n_i-4)\Theta_{\widehat{n_i}} +n_i\Phi_{\widehat{n_i}} -\Theta_{n_1, n_2,\cdots,n_t}}{\Theta_{n_1, n_2,\cdots,n_t}} \right]J_{n_i} + I_{n_i} = I_{n_i}.
\end{align*}

\noindent \underline{\textbf{Case 2:}} For $i\neq j$; $1\leq i,j \leq t$.
\begin{align*}
Y_{ij} &= 4(J_{n_i}-I_{n_i})  \left[
-\dfrac{\Phi_{\widehat{n_i,n_j}}}{\Theta_{n_1, n_2,\cdots,n_t}} J_{n_i \times n_j}\right] + J_{n_i \times n_j}  \left[ \left( \dfrac{3 \Theta_{\widehat{n_j}} + \Phi_{\widehat{n_j}}}{4 \Theta_{n_1, n_2,\cdots,n_t}} \right)J_{n_j} - \dfrac{1}{4} I_{n_j}
\right]\\
&\hspace*{9.5cm} - \sum_{k=1 \atop k \neq i,j}^t  \dfrac{\Phi_{\widehat{n_k,n_j}}}{\Theta_{n_1, n_2,\cdots,n_t}}  J_{n_i \times n_k} J_{n_k \times n_j}\\
&= \left[-  \dfrac{4n_i\Phi_{\widehat{n_i,n_j}}}{\Theta_{n_1, n_2,\cdots,n_t}} + \dfrac{4\Phi_{\widehat{n_i,n_j}}}{\Theta_{n_1, n_2,\cdots,n_t}}+  \dfrac{n_j (3 \Theta_{\widehat{n_j}} + \Phi_{\widehat{n_j}})}{4 \Theta_{n_1, n_2,\cdots,n_t}}- \dfrac{1}{4} 
 \right] J_{n_i \times n_j}- \sum_{k=1 \atop k \neq i,j}^t \dfrac{n_k  \Phi_{  \widehat{n_k,n_j}}}{\Theta_{n_1, n_2,\cdots,n_t}}J_{n_i \times n_j}\\
&= \left[-  \dfrac{3n_i\Phi_{\widehat{n_i,n_j}}}{\Theta_{n_1, n_2,\cdots,n_t}} + \dfrac{4\Phi_{\widehat{n_i,n_j}}}{\Theta_{n_1, n_2,\cdots,n_t}}+  \dfrac{ 3n_j \Theta_{\widehat{n_j}} +n_j \Phi_{\widehat{n_j}}}{4 \Theta_{n_1, n_2,\cdots,n_t}}- \dfrac{1}{4}  \right] J_{n_i \times n_j}  - \sum_{k=1 \atop k \neq j}^t \dfrac{n_k  \Phi_{  \widehat{n_k,n_j}}}{\Theta_{n_1, n_2,\cdots,n_t}}J_{n_i \times n_j}\\
&= \left[-  \dfrac{(3n_i-4)\Phi_{\widehat{n_i,n_j}}}{\Theta_{n_1, n_2,\cdots,n_t}} +  \dfrac{ 3n_j \Theta_{\widehat{n_j}} +n_j \Phi_{\widehat{n_j}}}{4 \Theta_{n_1, n_2,\cdots,n_t}}- \dfrac{1}{4}  \right] J_{n_i \times n_j}  - \dfrac{\Psi_{ \widehat{n_j}}  }{\Theta_{n_1, n_2,\cdots,n_t}}J_{n_i \times n_j}\\
&= \left[-  \dfrac{\Phi_{\widehat{n_j}}}{\Theta_{n_1, n_2,\cdots,n_t}} +  \dfrac{ 3n_j \Theta_{\widehat{n_j}} +n_j \Phi_{\widehat{n_j}}-\Theta_{n_1, n_2,\cdots,n_t}}{4 \Theta_{n_1, n_2,\cdots,n_t}} - \dfrac{\Psi_{ \widehat{n_j}}  }{\Theta_{n_1, n_2,\cdots,n_t}} \right] J_{n_i \times n_j} \\
&= \left[ -\dfrac{\Phi_{\widehat{n_j}}}{\Theta_{n_1, n_2,\cdots,n_t}} + \dfrac{3n_j \Theta_{\widehat{n_j}}+n_j \Phi_{\widehat{n_j}}-\Theta_{n_1, n_2,\cdots,n_t} }{4\Theta_{n_1, n_2,\cdots,n_t}} - \dfrac{\Theta_{\widehat{n_j}}-\Phi_{\widehat{n_j}}}{\Theta_{n_1, n_2,\cdots,n_t}}
\right]J_{n_i \times n_j} \\
&= \left[\dfrac{(3n_j-4)\Theta_{\widehat{n_j}} +n_j\Phi_{\widehat{n_j}} -\Theta_{n_1, n_2,\cdots,n_t}}{4\Theta_{n_1, n_2,\cdots,n_t}} \right]J_{n_i \times n_j} = \textbf{0}_{n_i \times n_j}.
\end{align*}
\end{proof}

In the next section, we will discuss  a few properties of the  Laplacian-like matrix $\mathcal{L}$ defined in Eqn.~\eqref{eqn:lap_like_sq_multi} and also investigate the inertia $\textup{In}(\mathcal{L})$.

\section{ A Few  Properties of the Laplacian-like Matrix }\label{sec:properties_of_L}

For a  complete $t$-partite graph $K_{n_1,n_2,\cdots,n_t}$ on $n=\sum_{i=1}^t n_i$ vertices, it is known that $0$ is an eigenvalue of the Laplacian-like matrix $\mathcal{L}$  defined in Eqn.~\eqref{eqn:lap_like_sq_multi}. In the next result, we prove that $0$ is a simple eigenvalue and also compute the cofactor of any two elements of $\mathcal{L}$ if $\det  \Delta(G) \neq 0 $.

\begin{theorem}\label{thm:rank_cof}
Let $G=(V,E)$ be a complete $t$-partite graph $K_{n_1,n_2,\cdots,n_t}$ on $n=\sum_{i=1}^n n_i$ vertices with $\cof \Delta(G) \neq 0$.  If $\mathcal{L}$ is the Laplacian-like matrix  defined in Eqn.~\eqref{eqn:lap_like_sq_multi}, then  {\textup{ rank}} $(\mathcal{L})= n-1$. Furthermore, if $\det  \Delta(G) \neq 0 $, then the cofactors of any two elements of $\mathcal{L}$ are equal to $ \dfrac{(-1)^{n-1}}{\cof \Delta(G)}$.
\end{theorem}

\begin{proof}
Let $G=(V,E)$ be a complete $t$-partite graph $K_{n_1,n_2,\cdots,n_t}$ on $n=\sum_{i=1}^n n_i$ vertices with $\cof \Delta(G) \neq 0$.  By Eqn.~\eqref{eqn:lap_like_sq_multi}, we have that Laplacian-like matrix $\mathcal{L}$ is a symmetric matrix and  $\mathcal{L}\mathds{1} = \mathbf{0}$ and $\mathds{1}^t \mathcal{L} = \mathbf{0}$. Thus,   $0$ is an eigenvalue of $\mathcal{L}$ and $\mathds{1}$ is a corresponding eigenvector. Furthermore, using $\mathcal{L}$ and $\Delta(G)$ are symmetric matrices,  $\mathds{1}$ is an  eigenvector of $\mathcal{L} \Delta(G)$ corresponding to the eigenvalue $0$.

Let $\mathbf{x}$ be an eigenvector of $\mathcal{L} \Delta(G)$ corresponding to the eigenvalue $0$, {\it i.e., } $\mathbf{x}^t \mathcal{L}\Delta(G) = 0$.  Assume that $\mathbf{x}$ is not in the span of $\{\mathds{1}\}$. By Lemma~\ref{lem:Lsq+I=nu}, we have  $\mathbf{x}^t(\mathcal{L}\Delta(G) + I)= \mathbf{x}^t\nu \mathds{1}^t$ and hence using $\mathbf{x}^t \mathcal{L}\Delta(G) = 0$, we get $\mathbf{x}^t = \mathbf{x}^t \nu \mathds{1}^t.$ Which is a contradiction to our assumption. Therefore,  $\textup{rank} (\mathcal{L})= n-1$.

Next, let $\det  \Delta(G) \neq 0 $. Then, by Theorem~\ref{thm:Delta_inv} we get $\Delta(G)^{-1} = - \mathcal{L} + \frac{1}{\lambda} \nu \nu^t$. Thus, using the determinant property $\det(A + uv^t) = \det(A) + v^tadj(A)u$, we have 
$$\det(\Delta(G)^{-1}) = \det(- \mathcal{L}) + \frac{1}{\lambda} \nu^t adj(- \mathcal{L}) \nu.$$
Using  $\textup{rank} (\mathcal{L})= n-1$, we have $\ds \det(\Delta(G)^{-1}) =  \frac{1}{\lambda} \nu^t adj(- \mathcal{L}) \nu = \frac{(-1)^{n-1}}{\lambda} \nu^t adj( \mathcal{L}) \nu$. Since $\mathcal{L}$ is a symmetric matrix and $\mathcal{L} \mathds{1} = 0$, using~\cite[Lemma~$4.2$]{Bapat}  the cofactors of any two elements of $\mathcal{L}$ are equal, say $c$. Then, 
\begin{equation}\label{eqn:cof-of-L}
\det(\Delta(G)^{-1}) = \frac{(-1)^{n-1}}{\lambda} \nu^t (cJ) \nu = \frac{(-1)^{n-1} c}{\lambda} \nu^t J \nu= \frac{(-1)^{n-1} c}{\lambda} \left(\sum_{v\in V} \nu(v)\right)^2.
\end{equation}
Using Eqn.~\eqref{eqn:nu_sq_multi}, we have
\begin{align}\label{eqn:nu=1}
\sum_{v\in V}\nu(v) &= \sum_{v\in V}\left [\frac{1}{\Psi_{n_1,n_2,\cdots,n_t}} \left( \sum_{i=1}^t \sum_{v \in V_i} \Phi_{\widehat{n_i}} \right)\right] \nonumber\\
&= \frac{1}{\Psi_{n_1,n_2,\cdots,n_t}}  \sum_{i=1}^t n_i \Phi_{\widehat{n_i}} \nonumber\\
&= \frac{1}{\Psi_{n_1,n_2,\cdots,n_t}}  \sum_{i=1}^t n_i \prod_{j=1\atop j\neq i}^t (3n_j-4)=1.
\end{align}
Substituting Eqns.~\eqref{eqn:lambda_sq_multi1} and \eqref{eqn:nu=1} in Eqn.~\eqref{eqn:cof-of-L}, we get $c=\dfrac{(-1)^{n-1}}{\cof \Delta(G)}$. This completes the proof.
\end{proof}

We now calculate a few eigenvalues of the Laplacian-like matrix $\mathcal{L}$.

\begin{prop}\label{prop:eigen_L}
Let $G$ be a complete $t$-partite graph $K_{n_1,n_2,\cdots,n_t}$ on $n=\sum_{i=1}^n n_i$ vertices with $\cof \Delta(G) \neq 0$ and $\mathcal{L}$ be the Laplacian-like matrix  defined in Eqn.~\eqref{eqn:lap_like_sq_multi}. Then, the following holds.
\begin{enumerate}
\item [$(i)$] $\dfrac{1}{4}$ is an eigenvalue of $\mathcal{L}$  with multiplicity  at least $n - t$. 

\item [$(ii)$]  If  $h=|\{i :n_i=1\}|$ and $h\geq 2$, then $1$ is an eigenvalue of $\mathcal{L}$ with multiplicity at least $h - 1$. 
\end{enumerate}
\end{prop}

\begin{proof}
Let $G = (V,E)$ be a complete $t$-partite graph $K_{n_1,n_2,\cdots,n_t}$  on $n=\sum_{i=1}^t n_i$ vertices such that the vertex $V$ is partitioned into $t$ subsets  $V_i$ for $1 \leq i \leq t$ and $|V_i| = n_i$. Suppose $h=|\{i :n_i=1\}|$ and $t=s+h$  such that  $n_i\geq 2$ for $1\leq i\leq s$ and $n_i=1$ for $s+1 \leq i\leq t=s+h$. 

Let $\mathbf{e}(p,q)$ be an $n$-dimensional column vector whose $p^{\text{th}}$ entry is $1$, $q^{\text{th}}$ entry is $-1$ and $0$ otherwise. Suppose the vertices of $G$ are indexed as in Eqn.~\eqref{eqn:D(K)}. Consider the  set of column vectors
\begin{align*}
\mathcal{E} = \{\mathbf{e}(1,j) \mid j=2,\cdots,n_1\} & \cup \{\mathbf{e}(n_1+1,n_1+j) \mid j=2,\cdots,n_2\}\\
&\cup \cdots \cup \{\mathbf{e} \left(\sum_{k=1}^{s-1}n_k+1,\sum_{k=1}^{s-1}n_k+j \right) \mid j=2,\cdots,n_s\}.
\end{align*}
Let $\mathcal{E}_1=\{\mathbf{e}(1,j) \mid j=2,\cdots,n_1\}$ and $\ds \mathcal{E}_i=\{\mathbf{e} \left(\sum_{k=1}^{i-1}n_k+1,\sum_{k=1}^{i-1}n_k+j \right) \mid j=2,\cdots,n_i\}$ for $i=2,3,\ldots,s$. Then $\ds \mathcal{E}=\cup_{i=1}^s \mathcal{E}_i$.  Using the definition of $\mathcal{L}$ as in Eqn.~\eqref{eqn:lap_like_sq_multi}, it is easy to see 
\begin{equation}\label{eqn:4.2.11}
\mathcal{L}\mathbf{x}= (a_i-b_i)\mathbf{x} \mbox{ for all } \mathbf{x}  \in\mathcal{E}_i; 1\leq i\leq s.
\end{equation}
For $ 1\leq i\leq s$, we have
\begin{align*}
a_i - b_i &= \dfrac{1}{2\Psi_{n_1,n_2,\cdots,n_t}}\left[ \dfrac{n_i-1}{2} \Theta_{\widehat{n_i}} +(n_i-3) \Psi_{\widehat{n_i}}\right] + \dfrac{1}{2\Psi_{n_1,n_2,\cdots,n_t}}\left[\dfrac{1}{2} \Theta_{\widehat{n_i}} +\Psi_{\widehat{n_i}}\right]\\
& = \dfrac{1}{2\Psi_{n_1,n_2,\cdots,n_t}}\left[ \dfrac{n_i}{2} \Theta_{\widehat{n_i}} +(n_i-2) \Psi_{\widehat{n_i}} \right]\\
& = \dfrac{1}{2\Psi_{n_1,n_2,\cdots,n_t}}\left[ \dfrac{n_i}{2} \Phi_{\widehat{n_i}}+ \dfrac{n_i}{2} \Psi_{\widehat{n_i}} +(n_i-2) \Psi_{\widehat{n_i}} \right]\\
& = \dfrac{1}{4\Psi_{n_1,n_2,\cdots,n_t}}\left[ {n_i} \Phi_{\widehat{n_i}}+(3n_i-4) \Psi_{\widehat{n_i}} \right] = \dfrac{1}{4}.
\end{align*}
Thus, from Eqn.~\eqref{eqn:4.2.11}, we get $\mathcal{L}\mathbf{x}= \dfrac{1}{4}\mathbf{x} \mbox{ for all } \mathbf{x}  \in\mathcal{E}$, and hence  $\dfrac{1}{4}$ is an eigenvalue of $\mathcal{L}$  with multiplicity  at least $|\mathcal{E}|=\ds \sum_{i=1}^s n_i -s$. Since $n_i=1$ for $s+1 \leq i\leq t=s+h$, $n-t= \ds \sum_{i=1}^t n_i -t=  \sum_{i=1}^s n_i -s$. This proves part  $(i)$.

To prove part  $(ii)$, let us assume $h=|\{i :n_i=1\}|$ and $h\geq 2$. Let $\widetilde{ \mathcal{E}}=\{\mathbf{e}(s+i,t) : i=1,2,\dots, h-1\}$. For $s+1\leq i < t=s+h$, using the definition of $\mathcal{L}$ as in Eqn.~\eqref{eqn:lap_like_sq_multi}, we have 
$$\mathcal{L}\mathbf{x}= (a_i-c_{it})\mathbf{x} \mbox{ for all } \mathbf{x}  \in\widetilde{ \mathcal{E}}.$$
Since $n_t=1$,  $\Phi_{\widehat{n_i}}= (3n_t-4)\Phi_{\widehat{n_i,n_t}} = - \Phi_{\widehat{n_i,n_t}}$. Thus, for  $s+1\leq i < t=s+h$,  we have
\begin{align*}
a_i - c_{it} &= \dfrac{1}{2\Psi_{n_1,n_2,\cdots,n_t}}\left[ \dfrac{n_i-1}{2} \Theta_{\widehat{n_i}} +(n_i-3) \Psi_{\widehat{n_i}}\right] - \dfrac{1}{\Psi_{n_1,n_2,\cdots,n_t}}\Phi_{\widehat{n_i,n_t}} \\
& = \dfrac{1}{2\Psi_{n_1,n_2,\cdots,n_t}}\left[  -2 \Psi_{\widehat{n_i}} - 2\Phi_{\widehat{n_i,n_t}} \right]\\
& = \dfrac{1}{\Psi_{n_1,n_2,\cdots,n_t}}\left[- \Psi_{\widehat{n_i}} - \Phi_{\widehat{n_i,n_t}} \right] \\
& =  \dfrac{1}{\Psi_{n_1,n_2,\cdots,n_t}}\left[- \Psi_{\widehat{n_i}} + \Phi_{\widehat{n_i}} \right]\\
& =\dfrac{1}{\Psi_{n_1,n_2,\cdots,n_t}}\left[(3n_i-4) \Psi_{\widehat{n_i}} + n_i\Phi_{\widehat{n_i}} \right] =1.
\end{align*}
Hence, $\mathcal{L}\mathbf{x}=\mathbf{x} \mbox{ for all } \mathbf{x}  \in\widetilde{ \mathcal{E}}$ and this completes the proof.
\end{proof}

Given a real symmetric matrix $M$ of order $n \times n$, we use the following convention where the eigenvalues of $M$ are in decreasing order:
\begin{equation}\label{eqn:ev-hermitian}
\lambda_1(M) \geq \lambda_2(M)\geq \cdots \geq \lambda_{n-1}(M)\geq  \lambda_{n}(M).
\end{equation}
We now state the Weyl’s inequality that gives interlacing inequalities of a rank one perturbation to a real symmetric matrix.
\begin{theorem}~\cite[Corollary~$4.3.9$]{Horn}\label{thm:interlacing}
Let $A$ and $B$ be real symmetric matrices of order $n \times n$ with eigenvalues ordered as in Eqn.~\eqref{eqn:ev-hermitian} such that $B = A+\alpha \alpha^t$, where $\alpha$ is a column vector. Then, 
\begin{equation*}
		\lambda_1(B) \geq \lambda_1(A) \geq \lambda_2(B) \geq \lambda_2(A) \geq \cdots \geq \lambda_n(B) \geq \lambda_n(A).
	\end{equation*}
\end{theorem}

In the following theorem, we compute the inertia of the  Laplacian-like matrix $\mathcal{L}$ subject to the condition $\det \Delta(K_{n_1,n_2,\cdots,n_t}) \neq 0$.

\begin{theorem}\label{thm:inertial_L}
Let $G$ be a complete $t$-partite graph $K_{n_1,n_2,\cdots,n_t}$ on $n=\sum_{i=1}^n n_i$ vertices with $\cof \Delta(G) \neq 0$ and  $\mathcal{L}$ be the Laplacian-like matrix  defined in Eqn.~\eqref{eqn:lap_like_sq_multi}. Then, the following holds.

\begin{itemize}
\item[$(i)$] If $n_i\geq 2$ for $1\leq i \leq t$, then $\textup{In}(\mathcal{L})= (n-t, 1 , t-1).$

\item[$(ii)$]Let $t=s+h$ and $h=|\{i :n_i=1\}|$. If $n_i\geq 2$ for $1\leq i \leq s$ and $n_i= 1$ for $s+1\leq i \leq s+h$, then

	\begin{equation}\label{eqn:inertia_lap}
		\textup{In}(\mathcal{L})=
		\begin{cases}
					(n-s-1,1,s) & \text{ if }  h-1 >\ds \sum_{i=1}^s \dfrac{n_i}{3n_i-4}, \\
			 
			(n-s,1,s-1) & \text{ if }  h-1 <\ds \sum_{i=1}^s \dfrac{n_i}{3n_i-4}. \\
		\end{cases}
	\end{equation}
\end{itemize}
\end{theorem}
\begin{proof}
Let $G$ be a complete $t$-partite graph $K_{n_1,n_2,\cdots,n_t}$ on $n=\sum_{i=1}^n n_i$ vertices and  $n_i\geq 2$ for $1\leq i \leq t$. From the proof of~\cite[Theorem 4.1]{JD3}, it is known that $-4$ is the only negative eigenvalue of the squared distance matrix $\Delta(G)$ with multiplicity of $n-t$ and $\textup{In}(\Delta(G))= (t, 0 , n-t).$  Thus, $\det \Delta(G) \neq 0 $,  $-\frac{1}{4}$ is the only negative eigenvalue of $\Delta(G)^{-1}$ with multiplicity of $n-t$, and $\textup{In}(\Delta(G)^{-1})= (t, 0 , n-t).$ Furthermore, under the assumption $\det \Delta(G) \neq 0$ and $\cof \Delta(G) \neq 0$,  using Theorem~\ref{thm:Delta_inv}, the inverse of $\Delta(G)$  can be written as  $$\Delta(G)^{-1} = - \mathcal{L} + \frac{1}{\lambda} \nu \nu^t.$$ 
Therefore,   using the fact that $0$ is a simple eigenvalue of $\mathcal{L}$ due to Theorems~\ref{thm:rank_cof} and~\ref{thm:interlacing}, we get
\begin{align*}
&\lambda_{t-1}(- \mathcal{L}) \geq \lambda_{t}(\Delta(G)^{-1})\geq \lambda_{t}(- \mathcal{L})\geq \lambda_{t+1}(\Delta(G)^{-1})\geq \lambda_{t+1}(- \mathcal{L}),\\
& \mbox{\ \ (+ve) } \quad  \qquad \mbox{\ \ (+ve) }\quad \qquad (0) \quad \qquad \mbox{\ \ (-ve) }  \quad \qquad \mbox{\ \ (-ve) }
\end{align*}
{\it{i.e.,}} $\lambda_{t}(\Delta(G)^{-1})>0  \textup{ and } \lambda_{t+1}(\Delta(G)^{-1})=-\frac{1}{4} <0  \textup{ implies that } \lambda_{i}(- \mathcal{L})>0 \textup{ for } 1\leq i\leq t-1 , \lambda_{t}(- \mathcal{L})=0 \textup{ and }\lambda_{i}(- \mathcal{L})=-\frac{1}{4} \textup{ for }t+1\leq i\leq n$.   
Hence $\textup{In}(- \mathcal{L})= (t-1, 1 , n-t)$ and this proves part $(i).$

Let $t=s+h$ and $h=|\{i :n_i=1\}|$ with $n_i\geq 2$ for $1\leq i \leq s$ and $n_i= 1$ for $s+1\leq i \leq s+h$. From part $(i)$ of Theorem~\ref{thm:det-cof-zero}, $\det \Delta(G) \neq 0$ if and only if $h-1 \neq \ds \sum_{i=1}^s \dfrac{n_i}{3n_i-4}$. Similar to part $(i)$, we prove part~$(ii)$  using $0$ is a simple eigenvalue of $\mathcal{L}$  and the form  $\Delta(G)^{-1} = - \mathcal{L} + \dfrac{1}{\lambda} \nu \nu^t$.

Let $ h-1 >\ds \sum_{i=1}^s \dfrac{n_i}{3n_i-4}$. From~\cite[ Lemma~$4.4$ and Theorem~$4.5$]{JD3}, $\lambda_{i}(\Delta(G)) >0$ for $1\leq i\leq s+1$,  $\lambda_{i}(\Delta(G))<0$ for $s+2\leq i\leq n$,  and $\textup{In}(\Delta(G))= (s+1, 0 , n-s-1)$. Using $0$ is a simple eigenvalue of $\mathcal{L}$ and Theorem~\ref{thm:interlacing}, we get
\begin{align*}
&\lambda_{s}(- \mathcal{L}) \geq \lambda_{s+1}(\Delta(G)^{-1})\geq \lambda_{s+1}(- \mathcal{L})\geq \lambda_{s+2}(\Delta(G)^{-1})\geq \lambda_{s+2}(- \mathcal{L}),\\
& \mbox{\ \ (+ve) } \quad  \qquad \mbox{\ \ (+ve) }\qquad \qquad (0) \qquad \qquad \mbox{\ \ (-ve) }  \quad \qquad \mbox{\ \ (-ve) }
\end{align*}
which implies that $\textup{In}(- \mathcal{L})= (s, 1 , n-s-1)$. Hence $In( \mathcal{L})=  (n-s-1,1,s)$.

Next, let  $ h-1 <\ds \sum_{i=1}^s \dfrac{n_i}{3n_i-4}$. From~\cite[ Lemma~$4.4$ and Theorem~$4.5$]{JD3}, $\lambda_{i}(\Delta(G)) >0$ for $1\leq i\leq s$,  $\lambda_{i}(\Delta(G))<0$ for $s+1\leq i\leq n$,  and $\textup{In}(\Delta(G))= (s, 0 , n-s)$. Using Theorem~\ref{thm:interlacing}, we get
\begin{align*}
&\lambda_{s-1}(- \mathcal{L}) \geq \lambda_{s}(\Delta(G)^{-1})\geq \lambda_{s}(- \mathcal{L})\geq \lambda_{s+1}(\Delta(G)^{-1})\geq \lambda_{s+1}(- \mathcal{L}).\\
& \mbox{\ \ (+ve) } \quad  \qquad \mbox{\ \ (+ve) }\quad \qquad (0) \qquad \qquad \mbox{\ \ (-ve) }  \quad \qquad \mbox{\ \ (-ve) }
\end{align*}
Hence $\textup{In}( \mathcal{L})=  (n-s,1,s-1)$. This completes the proof.
\end{proof}

\begin{cor}
Let $G$ be a complete $t$-partite graph $K_{n_1,n_2,\cdots,n_t}$ on $n=\sum_{i=1}^n n_i$ vertices with $\cof \Delta(G) \neq 0$ and  $\mathcal{L}$ be the Laplacian-like matrix  defined in Eqn.~\eqref{eqn:lap_like_sq_multi}. If $\det \Delta(G) \neq 0$, then the following holds.
\begin{enumerate}
\item [$(i)$] $\dfrac{1}{4}$ is an eigenvalue of $\mathcal{L}$  with multiplicity  $n - t$. 

\item [$(ii)$]  If  $h=|\{i :n_i=1\}|$ and $h\geq 2$, then $1$ is an eigenvalue of $\mathcal{L}$ with  multiplicity  $h - 1$. 
\end{enumerate}
\end{cor}
\begin{proof}
Let $G$ be a complete $t$-partite graph $K_{n_1,n_2,\cdots,n_t}$ on $n=\sum_{i=1}^n n_i$ vertices.  We complete the proof by considering the following cases.

Let $n_i\geq 2$ for $1\leq i \leq t$. From the proof of part $(i)$ of Theorem~\ref{thm:inertial_L}, it is follows that $\frac{1}{4}$ is an eigenvalue of $\mathcal{L}$ with multiplicity at most $n-t$. Thus, the result follows from Proposition~\ref{prop:eigen_L}.

Next, let $t=s+h$ with  $n_i\geq 2$ for $1\leq i\leq s$ and $n_i=1$ for $s+1 \leq i\leq t=s+h$. From part $(iii)$ of ~\cite[ Lemma~$4.4$]{JD3}, $-4$ is an eigenvalue of $\Delta(G)$ with multiplicity $n-t$ and $-1$ is an eigenvalue of $\Delta(G)$ with multiplicity $h-1$. Therefore, the result follows from the interlacing of eigenvalues of $\Delta(G)^{-1}$ and $-\mathcal{L}$ arguments used in the proof of  Theorem~\ref{thm:inertial_L}.
\end{proof}

We conclude the article with a conjecture. Given a complete $t$-partite graph $K_{n_1,n_2,\cdots,n_t}$, by part~$(i)$ of the Theorem~\ref{thm:det-cof-zero} $\det \Delta(K_{n_1,n_2,\cdots,n_t})=0$ if and only if $h-1 = \ds\sum_{i=1}^s \dfrac{n_i}{3n_i-4}$, where  $t=s+h$ with  $n_i\geq 2$ for $1\leq i\leq s$ and $n_i=1$ for $s+1 \leq i\leq t=s+h$. Based on the examples encountered during the preparation of this manuscript, we believe the following holds true. If $\det \Delta(K_{n_1,n_2,\cdots,n_t})=0$ and $\mathcal{L}$ be the Laplacian-like matrix  defined in Eqn.~\eqref{eqn:lap_like_sq_multi} then 
\begin{enumerate}
\item[$(i)$] $\frac{1}{4}$ and $1$ are the only positive eigenvalues of $\mathcal{L}$ with multiplicity $n-t$ and $h-1$, respectively.

\item[$(ii)$] The inertia $\textup{In}(\mathcal{L})=(n-s-1,1,s).$
\end{enumerate}

\vspace*{.4cm}

\noindent{\textbf{\Large Acknowledgements}} The authors thank the anonymous referees for a careful reading of the manuscript and for various suggestions. The first author  Joyentanuj Das is partially supported by the National Science and Technology Council in Taiwan (Grant ID: NSTC-111-2628-M-110-002) and  EPSRC Early Career Fellowship (EP/T00729X/1) by U.K. Research and Innovation in the U.K.\\ 

\noindent{\textbf{\Large Statements and Declarations}}\\

\noindent{\textbf{\large Conflict of Interest}}  The authors have  no competing interests to declare.\\


\small{

}

\begin{appendices}

\section{Proof of the Lemma~\ref{lem:identities}}\label{apn_1}
	Let $n_i\in \mathbb{N}$ for $1\leq i\leq t$ and  $t\geq 2$. We complete the proof with repeated application of Eqns.~\eqref{eqn:phi} -~\eqref{eqn:theta}. Thus,
	\begin{align*}
			\ds \Theta_{n_1,n_2,\cdots,n_t} &= \Phi_{n_1,n_2,\cdots,n_t} + \Psi_{n_1,n_2,\cdots,n_t}\\
			&= (3n_i-4)\Phi_{\widehat{n_i}} + \sum_{k=1}^t \left(n_k \prod_{j=1 \atop j\neq k}^{t} (3n_j-4)\right)\\
			&= (3n_i-4)\Phi_{\widehat{n_i}} + \sum_{k=1}^t n_k \Phi_{\widehat{n_k}}.
	\end{align*}	
This proves part $(a)$ of the lemma. Similarly, to prove part $(b)$, 
\begin{align*}
			\ds \Theta_{n_1,n_2,\cdots,n_t} &= \Phi_{n_1,n_2,\cdots,n_t} + \Psi_{n_1,n_2,\cdots,n_t}\\
			&= (3n_i-4)\Phi_{\widehat{n_i}} + \sum_{k=1}^t \left(n_k \prod_{j=1 \atop j\neq k}^{t} (3n_j-4)\right)\\
			&= (3n_i-4)\Phi_{\widehat{n_i}} + (3n_i-4)\sum_{k=1 \atop k \neq i}^t \left(n_k \prod_{j=1 \atop j\neq i,k}^{t} (3n_j-4)\right) + n_i \prod_{j=1 \atop j\neq i}^{t} (3n_j-4)\\
			&= (3n_i-4)\Phi_{\widehat{n_i}} + (3n_i-4)\Psi_{\widehat{n_i}} + n_i \prod_{j=1 \atop j\neq i}^{t} (3n_j-4)\\
			&= (3n_i-4) \left(\Phi_{\widehat{n_i}} + \Psi_{\widehat{n_i}}\right) + n_i \Phi_{\widehat{n_i}}\\
			&= (3n_i-4)\Theta_{\widehat{n_i}}+ n_i \Phi_{\widehat{n_i}}.
		\end{align*}
Next, 
\begin{align*}
			\ds \Psi_{n_1,n_2,\cdots,n_t} &= \sum_{k=1}^t \left(n_k \prod_{j=1 \atop j\neq k}^{t} (3n_j-4)\right)\\
			&= (3n_i-4) \sum_{k=1 \atop k \neq i}^t \left(n_k \prod_{j=1 \atop j\neq i,k}^{t} (3n_j-4)\right) + n_i \prod_{j=1 \atop j\neq i}^{t} (3n_j-4)\\
			&= (3n_i-4)\Psi_{\widehat{n_i}} + n_i\Phi_{\widehat{n_i}}.
		\end{align*}
This establish the part $(c)$. Finally, to prove part $(d)$, 		
\begin{align*}
			\ds \Psi_{\widehat{n_i}}
			&=\sum_{k=1\atop k\neq i}^t n_k  \Phi_{\widehat{n_i,n_k}}\\
			&= (3n_j-4) \sum_{k=1 \atop k \neq i,j}^t \left(n_k \prod_{l=1 \atop l\neq i,j,k}^{t} (3n_l-4)\right) + n_i \prod_{l=1 \atop l\neq i,j}^{t} (3n_l-4)\\
			&=(3n_j-4)\Psi_{\widehat{n_i,n_j}}+n_j \Phi_{\widehat{n_i,n_j}}.
		\end{align*}		
This completes the proof.	\hfill $\square$

\end{appendices}

\end{document}